\documentclass[11pt,letterpaper]{amsart}
\usepackage{amsfonts, amssymb, amscd, latexsym, graphicx, psfrag, color, float, enumitem}
\usepackage{xypic}
\usepackage[all]{xy}
\usepackage{mathrsfs}
\usepackage[mathscr]{euscript}

\usepackage{tikz}

\usepackage[hyphens]{url}

\usepackage{todonotes}
\usepackage{hyperref}

\usepackage[hyphenbreaks]{breakurl}

\usepackage{tkz-euclide}
\usepackage{caption}
\newtheorem{dummy}{dummy}[section]
\newtheorem{lemma}[dummy]{Lemma}
\newtheorem{theorem}[dummy]{Theorem}
\newtheorem{innercustomthm}{Theorem}
\newenvironment{customthm}[1]
{\renewcommand\theinnercustomthm{#1}\innercustomthm}
  {\endinnercustomthm}

\newtheorem{conjecture}[dummy]{Conjecture}
\newtheorem{corollary}[dummy]{Corollary}
\newtheorem{proposition}[dummy]{Proposition}
\newtheorem*{theorem*}{Theorem}
\newtheorem{definition}[dummy]{Definition}
\newtheorem*{definition*}{Definition}

\newtheorem{remark}[dummy]{Remark}
\newtheorem{remark*}{Remark}

\newcommand{\mHH}{\mathrm{HH}}

\newcommand{\bA}{\mathbb{A}}

\newcommand{\bG}{\mathbb{G}}

\newcommand{\bZ}{\mathbb{Z}}

\newcommand{\cC}{\mathcal{C}}
\newcommand{\cS}{\mathcal{S}}

\newcommand{\cO}{\mathcal{O}}

\newcommand{\cX}{\mathcal{X}}

\newcommand{\cE}{\mathcal{E}}
\newcommand{\cN}{\mathcal{N}}
\newcommand{\cHH}{\mathcal{HH}}
\newcommand{\cM}{\mathcal{M}}

\newcommand{\Aff}[1]{\mathrm{Aff}(#1)}
\newcommand{\Spec}{\mathrm{S}\mathrm{pec}\,}

\newcommand{\CAlg}{\mathrm{CAlg}}

\newcommand{\Mod}{\mathrm{Mod}}
\newcommand{\coMod}{\mathrm{coMod}}

\newcommand{\colim}{\mathrm{colim}}

\newcommand{\Hom}{\mathrm{Hom}}

 \newcommand{\twocell}[1]{\ar@{}[#1]^(.30){}="a"^(.70){}="b" \ar@{=>} "a";"b"}
 \newcommand{\ocell}[1]{\ar@{}[#1]^(.30){}="a"^(.70){}="b" \ar@{=} "a";"b"}
\newcommand{\QCoh}{\mathrm{QCoh}}

\begin{document}

\author[Scherotzke]{Sarah Scherotzke}
\address{Sarah Scherotzke, University of Luxembourg\\
6, avenue de la Fonte\\
L-4364 Esch-sur-Alzette\\
Luxembourg}
\email{\href{mailto:sarah.scherotzke@uni.lu}{sarah.scherotzke@uni.lu}}

\author[Sibilla]{Nicol\`o Sibilla}
\address{Nicol\`o Sibilla, SISSA\\ 
Via Bonomea 265\\ 34136 Trieste TS\\
Italy}
\email{\href{mailto:nsibilla@sissa.it}{nsibilla@sissa.it}}

\author[Tomasini]{Paolo Tomasini}
\address{Paolo Tomasini, University of Luxembourg\\
6, avenue de la Fonte\\
L-4364 Esch-sur-Alzette\\
Luxembourg}
\email{\href{mailto:paolo.tomasini@uni.lu}{paolo.tomasini@uni.lu}}

\title[Fourier--Mukai equivalences for formal groups and elliptic Hochschild homology]{Fourier--Mukai equivalences for formal groups and elliptic Hochschild homology}

%\subjclass[2010]{14F05, 19E08}
%\keywords{Semi-orthogonal decompositions, Kummer flat K-theory}

\begin{abstract}
This paper establishes a unifying framework for various forms of twisted Hochschild homology by comparing two definitions of elliptic Hochschild homology: one introduced by Moulinos--Robalo--To\"en and the other by Sibilla--Tomasini. Central to our approach is a new Fourier--Mukai duality for formal groups. We prove that when $\widehat{E}$ is the formal group associated to an elliptic curve $E$, the resulting $\widehat{E}$-Hochschild homology coincides with the mapping stack construction of Sibilla--Tomasini. This identification also recovers ordinary and Hodge Hochschild homology as degenerate limits corresponding to nodal and cuspidal cubics, respectively. Building on this, we introduce global versions of elliptic Hochschild homology over the moduli stacks of elliptic and cubic curves, which interpolate between these theories and suggest a universal form of TMF-Hochschild homology.
\end{abstract}

\maketitle

\tableofcontents
\setcounter{tocdepth}{3}

\section{Introduction}
In the beautiful paper \cite{MRT} the authors investigate the behaviour of Hochschild homology in mixed characteristics.  One of their discoveries is that, at least for affine derived schemes, it is possible to define twisted versions of Hochschild homology that depend on the choice of a one-dimensional formal group $\mathbb{G}$. This gives rise in particular to a notion of elliptic Hochschild homology, when  $\mathbb{G}$ is the completion of an elliptic curve $E$ at the identity section. In independent work, the two last named authors  gave a definition of elliptic Hochschild homology in \cite{sibilla2023equivariant}; see also the recent paper \cite{bouaziz2025elliptic} for a very interesting alternative construction which is equivalent to the one from \cite{sibilla2023equivariant}. One of the sources of interests of elliptic Hochschild homology as defined in \cite{sibilla2023equivariant} is that the authors prove that, when working over $\mathbb{C}$, it recovers the complexified (equivariant) elliptic cohomology of the analytification, when turning on a natural circle action. %, when   passing to Tate fixed points for a . % the actual elliptic cohomology of the analytification (up ) and this is also true equivariantly with respect to the action of a reductive group.

 One of the main goals of this article is showing that the definition of elliptic Hochschild homology given in \cite{MRT} and in \cite{sibilla2023equivariant} agree. The way we achieve this is by revisiting the construction from \cite{MRT} through the lenses of a Fourier--Mukai duality for formal groups which we introduce, and which we believe to be of independent interest.  This yields an alternative definition $\mathbb{G}$-Hochschild homology in the sense of \cite{MRT}, for any formal group $\mathbb{G}$, which is arguably more concrete. When $\mathbb{G}=\widehat{E}$ is the completion of an elliptic curve, this is the key ingredient in the comparison with the definition from \cite{sibilla2023equivariant}. The other technical input is Lurie's theory of Fourier--Mukai duality for dual abelian varieties over an arbitrary  
 $\mathbb{E}_\infty$-ring. 
 
Building on our results, we give a definition of TMF-Hochschild homology as a universal form of elliptic Hochschild homology living over the moduli space of elliptic curves;  or, in the tmf-variant, over the moduli of cubic curves. These theories interpolate between elliptic Hochschild homologies for different elliptic curves, and also incorporate degenerations to ordinary Hochschild homology and Hodge Hochschild homology, i.e. the graded algebra of global differential forms. The authors in \cite{MRT}, and Moulinos in \cite{moulinos2024filtered}, explore the rich geometry of filtrations of $\mathbb{G}$-Hochschild homology. The key object in their theory is the $\mathbb{G}$-filtered circle, a stack over $[\mathbb{A}^1/\mathbb{G}_m]$, which encodes a filtration on  $\mathbb{G}$-Hochschild homology whose associated graded is Hodge Hochschild homology; in fact in \cite{MRT} this is only done when $\mathbb{G}$ is equal to the completion of the multiplicative group, whereas the general picture was established by Moulinos in terms of a general theory of filtered formal groups. We argue that when $\mathbb{G}$ is the completion of either $\mathbb{G}_a$, 
 $\mathbb{G}_m$ or an elliptic curve $E$, it should be possible to model the  $\mathbb{G}$-filtered circle in terms of the moduli theory of cubic curves. A precise conjecture is stated in the last section of the paper.
 
 \subsection*{Our results} 
 We work over an arbitrary abelian commutative ring $R$. To explain our results, we need to briefly recall  the setting of \cite{MRT} and \cite{moulinos2024filtered}.  The definition of $\mathbb{G}$-Hochschild homology depends on  Cartier duality theory. Over a field this is classical \cite{cartier1962groupes}, and can be formulated as an anti-equivalence between the category of abelian affine group schemes and abelian formal groups. Over an arbitrary ring $R$ this is more subtle, but we can nonetheless associate to a one-dimensional abelian formal group $\mathbb{G}$ a one-dimensional affine group scheme 
 $\mathbb{G}^\vee$. Let $S^1_{\mathbb{G}}:=B\mathbb{G}^\vee$. The authors in \cite{MRT} define  the $\mathbb{G}$-Hochschild homology of an affine scheme $X$ as 
 $$
 \mathrm{HH}_*^{\mathbb{G}}(X) :=\mathcal{O}(\mathrm{Map}_{\mathrm{dSt}_R}(S^1_{\mathbb{G}}, X))
 $$
 where $\mathrm{Map}_{\mathrm{dSt}_R}$ denotes the mapping stack. If $\mathbb{G} = \widehat{\mathbb{G}_m}$ is the completion of the multiplicative group then $
 \mathrm{HH}_*^{\mathbb{G}}(X)$ is equivalent to ordinary Hochschild homology; whereas when $\mathbb{G}=\widehat{\mathbb{G}_a}$ is the completion of the additive group, we obtain Hodge Hochschild homology, i.e. 
 $$
 \mathrm{HH}_*^{\widehat{\mathbb{G}_a}}(X) \simeq \mathcal{O}(\mathbb{T}_X[-1]) \simeq \bigoplus_{i \in \mathbb{N}} \mathbb{L}^{\wedge^i}[i]
 $$
 Our first main result is a kind of Fourier--Mukai duality relating a formal group 
$\mathbb{G}$ and the $\mathbb{G}$-twisted affine circle 
$S^1_{\mathbb{G}}$.
 \begin{customthm}{A} 
\label{intromain1}
There is an equivalence of symmetric monoidal $\infty$-categories
$$\mathrm{F}_{\mathbb{G}}: \QCoh(\bG)^{\star}\simeq\QCoh(S^1_\bG)^\otimes$$
where $\star$ denotes the convolution tensor product, and $\otimes$ the ordinary tensor product. As a consequence, there is an equivalence 
$$
S^1_\bG \simeq    \mathrm{coSpec}( \mathrm{Hom}_{\QCoh(\bG)}(\mathcal{O}_e,\mathcal{O}_e)) 
$$
where $\mathrm{Hom}_{\QCoh(\bG)}(\mathcal{O}_e,\mathcal{O}_e)$ is the endomorphism algebra of the skyscraper sheaf at the identity $e \in \mathbb{G}$ equipped with its natural enhancement to an  $\mathbb{E}_\infty$-algebra. 
 \end{customthm}
 In the statement,  the symbol $\mathrm{coSpec}$ denotes the cospectrum of a coconnective commutative algebra; we refer the reader to the main text for more information on this important construction. 
 
Now, let $E$ be an elliptic curve over $R$. Following \cite{sibilla2023equivariant}, the elliptic Hochschild homology of an affine scheme $X$ over $R$ is defined as $$
 \mathcal{HH}_E(X) := \mathcal{O}(\mathrm{Map}_{\mathrm{dSt}_R}(E, X))
 $$
 Theorem \ref{intromain1} is the main ingredient in the proof of the following comparison result.
 \begin{customthm}{B} 
\label{intromain2}
 Let $\widehat{E}$ be the completion of $E$ at the identity. Then there is a natural equivalence
 $$
 \mathrm{HH}_*^{\widehat{E}}(X) \simeq  \mathcal{HH}_E(X)  
 $$
 Additionally, let $\mathcal{N}$ and $\mathcal{C}$ be respectively the nodal and cuspidal cubic curve over $R$. Then there are natural equivalences
 $$
 \mathrm{HH}_*(X) \simeq \mathrm{HH}_*^{\widehat{\mathbb{G}_m}}(X) \simeq \mathcal{O}(\mathrm{Map}_{\mathrm{dSt}_R}(\mathcal{N}, X)), \quad 
 \mathrm{HH}_*^{\widehat{\mathbb{G}_a}}(X) \simeq \mathcal{O}(\mathrm{Map}_{\mathrm{dSt}_R}(\mathcal{C}, X))
 $$
 \end{customthm}
Using Theorem \ref{intromain2}  we define global versions of elliptic Hochschild homology which interpolate  between different $\widehat{E}$-Hochschild homologies, as $E$ varies in a moduli stack. More precisely, we describe three variants of this construction corresponding to the three main deformations theories parametrising elliptic curves and their degenerations: namely the moduli stack of elliptic curves $\mathcal{M}_{ell}$, its modular  compactification $\overline{\mathcal{M}_{ell}}$, and the moduli stack of cubics $\mathcal{M}_{cub}$. We denote the resulting global Hochschild homology theories by 
 $$
  \mathcal{HH}_*^{TMF}(X), \quad \mathcal{HH}_*^{Tmf}(X), \quad \text{and} \quad \mathcal{HH}_*^{tmf}(X)
 $$
The notations are meant to be suggestive of the close link between these theories and the $R$-linearization of TMF and its variants, as studied in the homotopy theory literature.   
 \begin{customthm}{C} 
\label{intromain3}
 The  stalks of $ 
  \mathcal{HH}_*^{TMF}(X)$, $\mathcal{HH}_*^{Tmf}(X)$ and  $\mathcal{HH}_*^{tmf}(X)
 $ at a point $C$ of the corresponding moduli stack are equivalent to  
 \begin{itemize}
 \item $ \mathrm{HH}_*^{\widehat{E}}(X)$ if $C=E$ is an elliptic curve 
 \item $\mathrm{HH}_*^{\widehat{\mathbb{G}_m}}(X)$ if $C=\mathcal{N}$ is the nodal cubic curve
 \item and to $\mathrm{HH}_*^{\widehat{\mathbb{G}_a}}(X)$ if $C=\mathcal{C}$ is the cuspidal cubic curve.
 \end{itemize} 
 \end{customthm}
 
\subsection{Future work}
The theory of $\mathbb{G}$-Hochschild homology as developed in  \cite{MRT} and \cite{moulinos2024filtered} is meaningful in the affine setting, but breaks down when applied to more general geometric objects.  Informally, we could say that the twisted circle $S^1_{\mathbb{G}}$ is too small to see complicated global geometries, in particular  when  stackyness comes in to play. %the geometry of more complicated spaces than affine schemes, such as for instance quotient stacks. 
This is true already for $\mathbb{G}:=\widehat{\mathbb{G}_{m}}$. As we explained  
 $\mathrm{HH}^{\widehat{\mathbb{G}_m}}$ should coincide  with ordinary Hochschild homology, but in fact this holds only for affine schemes. In general, if $X$ is a stack
 $$
 \mathrm{HH}^{\widehat{\mathbb{G}_m}}_*(X) \simeq \mathcal{O}(\mathrm{Map}_{\mathrm{dSt}_R}(\mathrm{Aff}(S^1_{\mathrm{B}}), X)) \quad \text{and} \quad 
 \mathrm{HH}_*(X) \simeq \mathcal{O}(\mathrm{Map}_{\mathrm{dSt}_R}(S^1_{\mathrm{B}}, X))
 $$
 are genuinely different objects; in fact, under suitable assumptions, $\mathrm{HH}^{\widehat{\mathbb{G}_m}}_*(X)$ can be recovered as a completion of $ \mathrm{HH}_*(X)$.

 Theorem \ref{intromain2}  suggests a definition of $\mathbb{G}$-Hochschild homology when $\mathbb{G}$ is either $\widehat{\mathbb{G}_a}$, $\widehat{\mathbb{G}_a}$ or $\widehat{E}$ which is exempt from this limitation.  %This is keeping with the approach in  \cite{sibilla2023equivariant}, where the authors developed a theory of equivariant elliptic Hochschild homology in terms of the stack of (almost constant) maps from an elliptic curve. 
   In this paper we shall   place ourselves in the setting of  \cite{MRT} and \cite{moulinos2024filtered}, and therefore limit our attention to affine schemes. Forthcoming work of the third author with Forero Pulido \cite{ChristianPaolo} will explore extension of this story to the equivariant setting.  
   %, by generalizing the results of this paper to quotient stacks. 
   In particular, as already mentioned, one of the main results of \cite{sibilla2023equivariant} is a comparison (over $\mathbb{C}$) between the equivariant elliptic periodic cyclic homology and the complexified equivariant elliptic cohomology of the analytification. It is an interesting problem to generalize this picture to complexified equivariant TMF. This is accomplished in the forthcoming work \cite{ChristianPaolo} by generalizing the results of this paper to quotient stacks.

   {\bf Acknowledgments:}  The beginnings of this project go back to a week-long visit of Tasos Moulinos to SISSA in 2022. We thank Tasos for many inspiring discussions   and for his patient explanations of  \cite{MRT} and \cite{moulinos2024filtered}. We  thank Christian Forero-Pulido for many insightful exchanges around the topics of this paper, and for suggesting the approach to  Proposition \ref{nodalcusp}. Thanks also go to Emanuele Pavia, Mauro Porta and Rin Ray for helpful conversations.
\section{Comparison}
  
\subsection{Formal groups and Cartier duality}
We shall be interested in the definition proposed in \cite{MRT} of  $\mathbb{G}$--Hochschild homology, where $\mathbb{G}$ is an abelian formal group of dimension one. We will follow the account of  $\mathbb{G}$--Hochschild homology from \cite{moulinos2024filtered} as it works over a general base ring $R$; in fact, Moulinos also considers the setting where $R$ is not discrete   but rather it is an $\mathbb{E}_\infty$ ring spectrum, however we shall not need that level of generality here. 

We start by briefly recalling the basics of Cartier duality, which is the main technical input in the construction of $\mathbb{G}$--twisted Hochschild homology. The source of this material are sections 2 and 3 of \cite{moulinos2024filtered} to which we refer the reader for additional information.

We fix throughout a discrete commutative ring $R$. We denote by $\mathrm{CAlg}_R^\heartsuit$ the category of commutative discrete $R$-algebras, and by  $\mathrm{CAlg}_R^{\heartsuit, \mathrm{ad}}$ the category of  commutative discrete $R$-algebras equipped with the $I$-adic topology, for some ideal of definition $I \subset R$. There is a natural inclusion
$$
\mathrm{CAlg}_R^\heartsuit \to \mathrm{CAlg}_R^{\heartsuit, \mathrm{ad}}
$$
 If $A$ is an adic algebra, its formal spectrum is denoted $\mathrm{Spf}(A)$, and  is defined as the left Kan extension of the   functor $\mathrm{Spec}$ out of $\mathrm{CAlg}_R^\heartsuit$
 $$
 \xymatrix{
 \mathrm{Set}^{\mathrm{CAlg}_R^\heartsuit}   & \\
  \mathrm{CAlg}_R^\heartsuit \ar[u]^-{\mathrm{Spec}} \ar[r] & \mathrm{CAlg}_R^{\heartsuit, \mathrm{ad}} \ar@{-->}[ul]_-{\mathrm{Spf}} 
 }
 $$
We say that $\mathrm{Spf}(A)$ is an affine formal scheme. More generally, a formal scheme $X$ over $R$ is a functor
$$
X: \mathrm{CAlg}_R^\heartsuit \to \mathrm{Set}
$$
 which is Zariski locally of the form $\mathrm{Spf}(A)$. A \emph{formal group} is a group object in formal schemes.  In this paper we will be  concerned with abelian formal groups. These can be explicitly described as the datum of a formal scheme $\mathbb{G}$ which takes values in abelian groups, and that commutes with direct sums. We denote by $f \mathrm{Sch}$ the category of formal schemes, and by $\mathrm{Ab}(f \mathrm{Sch})$ the category of abelian formal groups.

Standard Cartier duality is a duality between commutative and cocommutative Hopf algebras over a field $k$, and affine abelian formal groups over $k$. The story over a general ground ring $R$ is more subtle. Again, we  follow Moulinos  account from \cite{moulinos2024filtered}. Recall that a 
commutative and cocommutative  Hopf algebra  over $R$ can be defined as an abelian group object in the category of cocommutative coalgebras over $R$: in symbols, it is an object in 
$$
\mathrm{Ab(cCAlg)}
$$
  The following definition is due to Lurie \cite{ELL2}.
\begin{definition}
Let $A$ be a commutative and cocommutative Hopf algebra over $R$. We say that $A$ is smooth if the following conditions are satisfied:
\begin{itemize}
\item the underlying $R$-module of $A$ is flat over $R$
\item as a coalgebra, it is isomorphic to the divided power coalgebra of a finitely-generated projective $R$-module $M$
\end{itemize}
We denote the category of smooth commutative and cocommutative Hopf algebras by 
$ \, 
\mathrm{Ab(cCAlg)}^{\mathrm{sm}}. 
$
\end{definition}
Let $\mathrm{Ab}(\mathrm{Aff}_R)$ be the category of  affine abelian groups over $R$. There is a fully-faithful functor 
$$
\Phi: \mathrm{Ab(cCAlg)}^{\mathrm{sm}} \to (\mathrm{Ab}(\mathrm{Aff}_R))^{\mathrm{op}}
$$
which sends a commutative Hopf algebra $A$ to the affine scheme  $\mathrm{Spec}(A)$ equipped with its natural group structure. 
\begin{proposition}[Construction 3.17, \cite{moulinos2024filtered}]
There is a fully-faithful functor 
$$
\Psi:  \mathrm{Ab(cCAlg)}^{\mathrm{sm}} \to  \mathrm{Ab}(f \mathrm{Sch})
$$
that sends a  commutative and cocommutative Hopf algebra $A$ to the functor
$$
\mathrm{Spec}^{\mathrm{c}}(A): \mathrm{CAlg}_R \to \mathrm{Ab}, \quad \mathrm{Spec}^{\mathrm{c}}(A)(S) = \{x \in A \otimes_R S | \Delta(x) = x \otimes x  \}
$$
\end{proposition}
We obtain two fully-faithful functors 
$$
(\mathrm{Ab}(\mathrm{Aff}_R))^{\mathrm{op}} 
\stackrel{\Phi} \leftarrow \mathrm{Ab(cCAlg)}^{\mathrm{sm}} \stackrel{\Psi} \rightarrow 
\mathrm{Ab}(f \mathrm{Sch})
$$
Let us comment briefly on the main differences between the  general $R$-linear set-up, and the classical setting of Cartier duality over a field $k$. In the latter case, there is no need to restrict to smooth commutative and cocommutative Hopf algebras. The functors $\Phi$ and $\Psi$ are defined over the whole category of commutative and cocommutative Hopf algebras, and they are both equivalences. In that setting, Cartier duality is precisely the anti-equivalence 
$$\Phi \circ \Psi^{-1}: \mathrm{Ab}(f \mathrm{Sch}) \to (\mathrm{AffAb}_k)^{\mathrm{op}}$$

Over a general base $R$, the functor $\Psi$ is only fully-faithful, so we cannot take its inverse.  
For our purposes, however, it is sufficient that the composite $\Phi \circ \Psi^{-1}$ is defined over a particularly simple class of objects 
in $\mathrm{Ab}(f \mathrm{Sch})$. 
Namely, we are interested in one dimensional abelian formal groups, i.e. objects in $\mathrm{Ab}(f \mathrm{Sch})$ whose underlying formal scheme is isomorphic to $\mathrm{Spf}(R[[t]])$. Let us denote the sub-category they form by $\mathrm{Ab}(f \mathrm{Sch})^{\mathrm{1}-dim}$. It is easy to see that $\mathrm{Ab}(f \mathrm{Sch})^{\mathrm{1}-dim}$ lies in the essential image of $\Psi$. Let us set  
$$
(-)^\vee:=\Phi \circ \Psi^{-1}: \mathrm{Ab}(f \mathrm{Sch})^{\mathrm{1}-dim} \to (\mathrm{AffAb}_k)^{\mathrm{op}}
$$
Unpacking the construction, one sees  that  there is an isomorphism
$$
\mathcal{O}(\bG^\vee) \cong \mathcal{O}(\bG)^*
$$
where $(-)^*$ denotes the $R$-linear dual. 

We are now ready to define $\mathbb{G}$-Hochschild homology. This requires briefly fixing our conventions concerning derived geometry over $R$, even if the specifics only play a minor  role in the argument. We refer the reader to Section 1.2 of \cite{MRT} for additional details; we remark that occasionally,  for reasons of internal consistency,  our notations will diverge slightly from the ones in \cite{MRT}.

 Let $\mathrm{SCR}_R$ be the $\infty$-category of simplicial commutative $R$-algebras, and let $\mathrm{CAlg}_R^{\mathrm{cn}}$ be the category of connective $\mathbb{E}_\infty$-$R$-algebras. There is a monadic and comonadic monoidal functor relating the two (the Dold-Kan functor), which however fails to be an equivalence when $R$ is not a field. Thus there are effectively two distinct options to set-up the theory of  derived affine schemes over $R$.   The authors in \cite{MRT}  take derived affine schemes to be the opposite category of  set $\mathrm{SCR}_R$; we find it more convenient, however, to work instead with $\mathrm{CAlg}_R^{\mathrm{cn}}$. As explained in  \cite{moulinos2024filtered} these two approaches to 
$R$-linear derived  geometry, though different, are largely parallel; and in many situations one can freely switch between the two.

 We set 
$ \, \mathrm{dAff}_R:=(\mathrm{CAlg}_R^{\mathrm{cn}})^{\mathrm{op}}$. 
We denote by $\mathrm{St}_R$ the $\infty$-category of classical stacks, i.e. functors over a site of discrete $R$-algebras. We denote by $\mathrm{dSt}_R$ the $\infty$-category of derived stacks over $R$. Left Kan extension yields a functor
$$
L:\mathrm{St}_R \to \mathrm{dSt}_R
$$
If $\mathcal{X}$ and $\mathcal{Y}$ are derived stacks, we denote by $\mathrm{Map}_{\mathrm{dSt}_R}(\mathcal{X}, \mathcal{Y})$ their mapping stack.

Let us  also briefly introduce coaffine stacks, as they play a role in this story. Coaffine stacks were first studied  by To\"en in \cite{toen2006champs}, see also Lurie's treatment in \cite{DAGVIII}. They are associated with coconnective $\mathbb{E}_\infty$ $R$-algebras, as opposed to connective ones. %, which correspond instead to affine derived schemes. 
More precisely
%, as explained in Section 1.2 of \cite{MRT}, 
we can define a functor 
$$
\mathrm{coSpec}: (\mathrm{CAlg}_R^{\mathrm{cocn}})^{\mathrm{op}} \to \mathrm{St}_R. 
$$
The functor  admits a left adjoint, which is   denoted
$$
\mathrm{C}^* (-, \mathcal{O}): \mathrm{St}_R \to (\mathrm{CAlg}_R^{\mathrm{cocn}})^{\mathrm{op}}
$$ 
 By a small abuse of notation we denote by $\mathrm{coSpec}$ also the functor landing in $\mathrm{dSt}_R$, which is obtained by post-composing $\mathrm{coSpec}$ with $L$. We say that a derived stack is coaffine if it lies in the essential image of $\mathrm{coSpec}$. If $X$ is an object of $\mathrm{St}_R$
 we say that its affinization is the coaffine stack 
 $$\mathrm{Aff}(X):=\mathrm{coSpec} (\mathrm{C}^*(X, \mathcal{O}))$$  
\begin{definition}[Section 6.3 \cite{MRT}, Definition 7.2 \cite{moulinos2024filtered}]
Let $\mathbb{G}$ be an object in $\mathrm{Ab}(f \mathrm{Sch})^{\mathrm{1}-dim}$. We define $\mathbb{G}$-Hochschild homology to be the functor 
$$
\mathrm{HH}^{\mathbb{G}}_* : \mathrm{CAlg}_R^{\mathrm{cn}} \to \mathrm{CAlg}_R^{\mathrm{cn}} \, ,   \quad A \mapsto \mathrm{HH}^{\mathbb{G}}_*(A):=\mathcal{O}(\mathrm{Map}_{\mathrm{dSt}_R}(B \mathbb{G}^\vee, \mathrm{Spec}(A)))
$$ 
\end{definition}
The stack $B \mathbb{G}^\vee$ is always coaffine. This is the main reason why  the applicability of the theory of $\mathbb{G}$-Hochschild homology is ultimately  limited, as it is meaningful only in the context derived affine schemes; it is clear that for more general stacks the one above cannot be the correct definition. To explain this point, it will be helpful to record two interesting special cases of the definition above. %, which we shall revisit in the next section from a different perspective. An important source of formal groups is given by completions of algebraic groups at their identity elements. 
\begin{itemize}
\item Let $\mathbb{G} :=\widehat{\mathbb{G}_a}$ be the formal group obtained by completing the additive group $\mathbb{G}_{a}$ at $0$.   Then $$
B \mathbb{G}^\vee \simeq \mathrm{coSpec}(R \oplus R[\eta])
$$
where $\eta$ is in (cohomological) degree one, and $R \oplus R[\eta]$ is a  square-zero extension of $R$. In particular
$$
\mathrm{HH}^{\mathbb{G}}_*(A) \simeq \mathcal{O}(\mathbb{T}_{\mathrm{Spec}(A)}[-1])
$$
where $\mathbb{T}_{\mathrm{Spec}(A)}[-1]$ denotes the total space of the shifted tangent bundle of $\mathrm{Spec}(A)$.
\item Let $\mathbb{G}:=\widehat{\mathbb{G}_{m}}$ be the formal group obtained by completing the multiplicative group at $1$. Then
$$
B \mathbb{G}^\vee \simeq \mathrm{Aff}(S^1_\mathrm{B})
$$ 
where $S^1_\mathrm{B}$ denotes the Betti stack of $S^1$. Since $\mathrm{Spec}(A)$ is affine, we have that 
\begin{equation}
\label{affinization}
\mathcal{L}\mathrm{Spec}(A):=\mathrm{Map}_{\mathrm{dSt}_R}(S^1_\mathrm{B}, \mathrm{Spec}(A)) \simeq \mathrm{Map}_{\mathrm{dSt}_R}(\mathrm{Aff}(S^1_\mathrm{B}), \mathrm{Spec}(A))
\end{equation}
Thus in this case $\mathbb{G}$-Hochschild homology  coincides with ordinary Hochschild homology, as the latter can be defined as functions on the derived loop space, that is
$$
\mathrm{HH}^{\mathbb{G}}_*(A) \simeq \mathcal{O}(\mathcal{L} \mathrm{Spec}(A)) \simeq \mathrm{HH}_*(A)
$$ 
\end{itemize}

\subsection{Fourier--Mukai duality for formal groups}
 
In this section we establish a Fourier--Mukai duality for formal groups. Let $\mathrm{Mod}_R^\heartsuit$ be the heart of the standard t-structure on $\mathrm{Mod}_R$.

%
%Ma (https://ncatlab.org/nlab/show/adjoint+monad Proposition 5.1) se l'endofuntore definito da una comonade è a sua volta l'aggiunto destro dell'endofuntore definito da una monade (in questo caso: C⊗_R - --| C*⊗_R -), allora le categorie di moduli per la prima monade è equivalente alla categoria dei comoduli per la comonade, in maniera compatibile con i dimenticanti.
%
%In particolare, un C*-comodulo M viene mandato nel C-modulo che come R-modulo è sempre M, e l'azione 
%
%C⊗M --> M 
%
%è data dall'aggiunto del morfismo di R-moduli M --> M⊗C* = Hom_R(C,M).
%
%
%
%Sì, diciamo che ti serve un analogo derivato della Prop. 5.1. In generale, è ovvio che il dimenticante CoMod_C --> Mod_R è anche monadico (quello a prescindere di com'è fatta C). Quindi la cosa immediatamente non ovvia è dimostrare che
%1. quando C è perfetta come R-modulo, l'aggiunto sx del dimenticante nella comonade è dato da Hom_R(C, -) = C* ⊗ -.
%2. c'è un funtore CoMod_C --> Mod_C* che commuta con i dimenticanti (forse qua smanettando un po' sulla definizione delle algebre/coalgebre per l'operad LM e la sua duale si fa piuttosto direttamente).
%a quel punto Corollary 7.3.4.19 di Higher Algebra ti fa concludere l'equivalenza di monadi.

\begin{lemma} \label{lem:dual}
Let $A$ be an algebra object in $\mathrm{Mod}_R$ such that $A$ lies in $\mathrm{Mod}_R^\heartsuit$ and  is a projective  $R$-module. Let $A^*$ be the $R$-linear dual of $A$ with its natural coalgebra structure. Then there is an equivalence of categories
$$\coMod_{A^*} \simeq   \Mod_{A}$$
\end{lemma}
 
\begin{proof}
Since $A$ is a projective  $R$-module, it is  dualizable and therefore the internal-Hom functor 
$$\Hom_R(A, -):   \Mod_R \to   \Mod_R $$ is equivalent to the functor $A^* \otimes_R-$. As $A^*$ and $A$ are dual,  the functors $A^* \otimes_R-$ and  $A \otimes_R-$  form an ambidexterous adjunction. Now $A \otimes_R-$ coincides with the  endofunctor  of $\Mod_R$ which underlies the monadic forgetful functor  
$\Mod_A  \to \Mod_R$.  This is left-adjoint to the endofunctor of $\Mod_R$ underlying the comonadic forgetful functor 
$
\coMod_{A^*}  \to \Mod_R, 
$ as the latter is given by $A^* \otimes_R-$. So we have in fact an adjoint pair of a monad and a comonad: under this assumption it follows from the 
%$\infty$-categorical version of the
  Eilenberg--Moore theorem that there exists an equivalence 
$$
\Mod_A \simeq \coMod_{A^*}.
$$
% \ref{maclane2012sheaves} Chapter V.5 Theorem 2. 
\end{proof}

\begin{proposition} 
\label{prop:formalFM}
Let $\bG$ be an abelian formal group of dimension one over $R$. Then there is an equivalence of symmetric monoidal $\infty$-categories
$$\mathrm{F}_{\mathbb{G}}: \QCoh(\bG)^{\star}\simeq\QCoh(S^1_\bG)^\otimes$$
where $\star$ denotes the convolution tensor product induced by the formal group structure on $\bG$, and $\otimes$ denotes the ordinary tensor product on quasi-coherent sheaves.
\end{proposition}
\begin{proof}
We will prove that the following chain of equivalences is satisfied

$$\QCoh(\bG) \stackrel{(1)} \simeq \coMod_{\cO(\bG)^\ast} \stackrel{(2)} \simeq \coMod_{a^\ast a_\ast} \stackrel{(3)} \simeq \QCoh(B\Spec(\cO(\bG ^\vee))),$$ where $a: \Spec R \to B \mathbb{G}^\vee$. 
We start by showing $(1)$. Let 
$$
\bG \simeq \colim_i G_i \simeq \colim_i \Spec \cO(G_i)
$$ be a presentation of $\bG$ as an ind-affine scheme. Using descent for quasi-coherent sheaves we obtain equivalences
$$\QCoh(\bG)\simeq \lim_i \QCoh(G_i)\simeq\lim_i \Mod_{\cO(G_i)}$$

Note that for all $i$, we have that $\mathcal{O}(G_i)$ is a finite rank projective $R$-module; in fact, we can assume that $G_i$ is isomorphic to $\mathrm{Spec}(R[T]/T^i)$. 
By Lemma \ref{lem:dual} we obtain an equivalence 
$$\Mod_{\cO(G_i)}\simeq\coMod_{\cO(G_i)^*}$$ where $\cO(G_i)^*$ is the $R$-linear dual of $\cO(G_i)$. Further 
$$\lim_i \coMod_{\cO(G_i)^*}\simeq\coMod_{\cO(\bG)^*}$$
Hence, as we claimed we can  identify $\QCoh(\bG)$  with $\coMod_{\cO(\bG)^*},$  which is the category of comodules over the Hopf algebra of distributions on $\bG$, $$\mathrm{Dist}(\bG):=\cO(\bG)^*$$  

Next we turn to equivalence $(2)$. We need to show that the coaction of the comonad $a^\ast a_\ast$ is given by the tensor product by $\cO(\bG)^*$. We use base-change along the diagram
$$
\xymatrix{
\bG^\vee\ar[d]\ar[r] & \Spec R \ar[d]^{a} \\
\Spec R \ar[r]^{a} & B\bG^\vee
}
$$
In order to perform base-change, we need to show that the map $a:\Spec R \to B\Spec\cO(\bG)^\vee$ is perfect in the sense of Ben-Zvi--Francis--Nadler \cite{ben2010integral}; the fact that this is sufficient to have base-change is proved in Proposition 3.10 \cite{ben2010integral}. For this, it is enough to show that for any affine derived scheme $U$, the pullback of $a$ to $U$ is a perfect stack. This is immediate, as the map
$$
a: \Spec R \to B\bG^\vee 
 $$
 is affine: indeed, it is the atlas of the classifying stack of the affine group scheme $\bG^\vee$.   
 
Finally, let us prove equivalence $(3)$. Since $\bG^\vee$  is an affine  group scheme the functor 
$$a^\ast:\QCoh(B\bG^\vee)=\QCoh(B\Spec (\cO(\bG)^*))\to\QCoh(\Spec R)$$ 
is comonadic, and yields the desired identification  
$$\QCoh(B\Spec(\cO(\bG)^*))\simeq\coMod_{a^*a_*}$$
We conclude since, by definition, we have that 
$$S^1_\bG:=B\bG^\vee=B\Spec((\cO(\bG))^*)$$

Finally, let us turn to the monoidal structures. The fact that $\mathrm{F}_{\mathbb{G}}$ is a symmetric monoidal equivalence is not difficult: the two tensor structures are identified with tensor product over $R$ in $\coMod_{\cO(\bG)^\vee}$. However, let us spell  out the argument in more detail. The functors appearing in equivalences $(1)$ and $(3)\circ (2)$ are naturally symmetric monoidal. In particular, the equivalence $(3)\circ (2)$ is naturally symmetric monoidal as it comes from the symmetric monoidal functor 
$$a^\ast:\QCoh(B\bG^\vee)^\otimes\to\QCoh(\Spec R)^{\otimes_R}$$
where $\otimes_R$ is the symmetric monoidal structure given by tensor product relative to $R$. This implies that the induced functor $\phi^\otimes$ 
$$
\xymatrix{
    \QCoh(B\bG^\vee)^\otimes \ar[d]_{\phi^\otimes}\ar[r]^{a^\ast} & \QCoh(\Spec R)^{\otimes_R} \\
    \coMod_{\cO(\bG)^\ast}^{\otimes_R} \ar[ur]_{obv}
    }
$$
is also symmetric monoidal. Indeed, the functor $a^\ast$ naturally factors via $\coMod_{a^\ast a_\ast}$, and by the arguments above the comonad $a^\ast a_\ast$ acts as tensor product by $\cO(\bG)^\ast$. 

The case of equivalence $(1)$ is similar: we apply the same argument to the symmetric monoidal functor 
$$p_+:\QCoh(\bG)^\star\to\Mod_R^{\otimes_R}$$
defined as the left adjoint to the pullback functor
$$p^\ast:\Mod_R\to\QCoh(\bG)$$
along the structure map
$$p:\bG\to\Spec R$$
Indeed, the Hopf algebra of distributions $\mathrm{Dist}(\bG):=\cO(\bG)^\ast$ on $\bG$ is equivalently described as 
$$\mathrm{Dist}(\bG)=p_+ p^\ast (1_{\Mod_R})=p_+ p^\ast (R)=p_+(\cO_\bG)=\cO(\bG)^\ast$$

Before concluding the proof, let us explain why the functor $\phi^\otimes$ is symmetric monoidal. By \cite{Torii}, the forgetful functor $obv$ is (strong) symmetric monoidal.
%Moreover, it is also conservative. By the commutativity of the diagram above, we have
%$$a^\ast=obv\circ \phi$$
%and $a^ast$ is symmetric monoidal. Thus, 
Moreover, as $a^\ast$ is symmetric monoidal, the pushforward functor $a_\ast$ is lax monoidal. In particular, this implies that the right adjoint to $\phi$ is lax monoidal, as this functor is just the composition $a_\ast\circ obv$. Thus, $\phi$ is oplax monoidal, as it is left adjoint to a lax monoidal functor. We only need to check the property that this lax monoidal functor is symmetric monoidal: this is done by recalling that $obv \circ \phi=a^\ast$ is symmetric monoidal, and by using that $obv$ is conservative and symmetric monoidal itself.

%We observe the equivalence is functorial with respect to maps of formal groups $f:\bG\to\bH$. The pushforward along $f$ is 
%$$f_\ast:\QCoh(\bG)\to\QCoh(\bH)$$
%The equivalence $(1)$ is functorial with respect to $f$: taking global sections, $f$ induces a map
%$$\cO(\bH)\to\cO(\bG)$$
%whose dual is a map of \textcolor{red}{Hopf algebras}
%$$\cO(\bG)^\vee \to\cO(\bH)^\vee$$
%Then, the induced comodule structure gives a functor
%$$\coMod_{\cO(\bG)^\vee}\to\coMod_{\cO(\bH)^\vee}$$
%which corresponds to pushforward of quasi-coherent sheaves along $f$.
%
%The composition of $(2)$ and $(3)$ is also functorial: 
\end{proof}
If the formal group $\bG$ is the formal additive group $\widehat{\bG}_a$, a similar statement is proved in characteristic 0 in \cite[Proposition 5.17]{SafNaef}.
\begin{remark}
Proposition \ref{prop:formalFM} is closely related to the duality theory for commutative group stacks expounded by Arinkin (who credits it to Beilinson)  in \cite{donagi2008torus}[Appendix A]. If $\mathcal{X}$ is a commutative group stack, then its dual $\mathcal{X}^{\vee}$ is the stack of morphisms of commutative group stacks between $\mathcal{X}$ and $B\mathbb{G}_m$,
$$
\mathcal{X}^{\vee} := \mathrm{Map}_{\mathrm{Ab}(\mathrm{dSt}_R)}(\mathcal{X}, B\mathbb{G}_m)
$$
Under these assumptions, there should be a Fourier--Mukai equivalence 
$$
\mathrm{QCoh}(\mathcal{X}) \simeq \mathrm{QCoh}(\mathcal{X}^{\vee})
$$
that generalizes the classical Fourier--Mukai equivalence for dual abelian varieties over a field. 
Arinkin's appendix contains no proofs, and his set-up is somewhat more limited than the one we would need here. Arinkin does mention that the theory also applies to ind-affine groups but only under the assumption, which is too restrictive for us, that $R=\mathbb{C}$.

To bridge the gap between Arinkin's picture and our Proposition \ref{prop:formalFM}, the first step would be proving that in our setting there is an equivalence 
$$S^1_\bG\simeq\mathrm{Map}_{\mathrm{Ab}(\mathrm{dSt}_R)}(\bG,B\bG_m)$$
This would show that $S^1_\bG$ is indeed the dual commutative group stack of $\bG$. For the moment we leave this as a conjecture; we limit ourselves to point out that there is a natural comparison map   between the two stacks
$$
S^1_\bG := B \bG^\vee \simeq B \mathrm{Map}_{\mathrm{Ab}(\mathrm{dSt}_R)}(\bG,\bG_m) \simeq B \Omega_* \mathrm{Map}_{\mathrm{Ab}(\mathrm{dSt}_R)}(\bG,B\bG_m) 
\stackrel{c} \longrightarrow 
 \mathrm{Map}_{\mathrm{Ab}(\mathrm{dSt}_R)}(\bG,B\bG_m) 
$$
 which induces an equivalence on completions. 
\end{remark}

As a consequence of  Proposition \ref{prop:formalFM} we obtain the following  more explicit description of $S^1_\bG$. Let $\mathcal{O}_e$ be the skyscraper sheaf at the closed point 
$e \in \bG$. The $\mathbb{E}_1$-algebra $\mathrm{Hom}_{\QCoh(\bG)}(\mathcal{O}_e,\mathcal{O}_e)$ has in fact a natural upgrade to an $\mathbb{E}_\infty$-algebra. Indeed,  the object $\mathcal{O}_e$ is the unit for the symmetric tensor product on $\QCoh(\bG)^\star$. Further, $\mathrm{Hom}_{\QCoh(\bG)}(\mathcal{O}_e,\mathcal{O}_e)$ is coconnective; this is easily seen by computing Ext-modules via Koszul resolution.  Thus we have that $$
 \mathrm{Hom}_{\QCoh(\bG)}(\mathcal{O}_e,\mathcal{O}_e) \in \mathrm{CAlg}_R^{\mathrm{cocn}}
$$
\begin{corollary}\label{corollary:ellcircle}
There is an equivalence 
$$S^1_\bG \simeq \mathrm{coSpec}( \mathrm{Hom}_{\QCoh(\bG)}(\mathcal{O}_e,\mathcal{O}_e))$$
\end{corollary} 
\begin{proof}
The stacks $S^1_{\bG}$ is coaffine, hence we only need to identify the global sections of the respective structure sheaves.  Proposition \ref{prop:formalFM} gives us such identification. Since $\mathrm{F}_{\mathbb{G}}$ is symmetric monoidal it preserves unit objects, that is
$$
\mathrm{F}_{\mathbb{G}}(\cO_{S^1_{\bG}}) \simeq \mathcal{O}_e
$$
Thus $\mathrm{F}_{\mathbb{G}}$ induces an equivalence of objects in $\mathrm{CAlg}_R^{\mathrm{cocn}}$
$$\cO(S^1_{\bG})\simeq\mathrm{Map}_{\QCoh(S^1_{\bG}) }(\cO_{S^1_{\bG}},\cO_{S^1_{\bG}})\simeq \mathrm{Map}_{\QCoh(\bG)}(\mathcal{O}_e,\mathcal{O}_e).$$
\end{proof}
The last ingredient we shall need is the theory of Fourier--Mukai duality for dual elliptic curves over a general discrete base ring $R$. Most references on this topic assume that $R$ is a field. However Lurie \cite{ELL1} develops the theory in great generality over an $\mathbb{E}_\infty$-ring spectrum. Recall that for us, here as elsewhere in the paper, $R$ is in fact a discrete commutative ring and this entails some simplifications in Lurie's treatment. 

We can summarize the relevant points from Lurie's work as follows. Let $E$ be an elliptic curve over $R$. It is possible to define the dual of $E$, either as the connected component of the identity in the Picard scheme of $E$, $\mathcal{P}\mathrm{ic}^0_E$ (for the precse definition see Construction 5.4.6 \cite{ELL1}); or as the moduli space of multiplicative line bundles on $E$, $\mathcal{P}\mathrm{ic}^m_E$ (see Definition 5.3.1 \cite{ELL1}). When $R$ is a discrete ring  these two objects agree, and are isomorphic to $E$. In Lurie's picture we get a perfect bi-extensions
$$
E \wedge E \to \mathrm{BGL}_1
$$
which classically corresponds to the datum of the Poincar\`e bundle,  which is the kernel yielding the Fourier--Mukai equivalence. Let $\mathrm{QCoh}(E)$ be the category of quasi-coherent sheaves on $E$. \begin{theorem}[Proposition 5.1.3 \cite{ELL1}]
\label{FML}
There is a symmetric monoidal equivalence 
$$
\mathrm{FM}_E: \mathrm{QCoh}(E)^\star \to \mathrm{QCoh}(E)^\otimes
$$
where $\star$ and $\otimes$ denote, respectively, the convolution tensor product and the ordinary tensor product. 
\end{theorem}
Let $\widehat{E}$ in $\mathrm{Ab}(f \mathrm{Sch})^{\mathrm{1}-dim}$ be the completion of an elliptic curve $E$ over $R$ at the identity section. 
\begin{corollary}
\label{comparison}
There is an equivalence of coaffine stacks
$$
\mathrm{Aff}(E) \simeq S^1_{\widehat{E}}
$$
\end{corollary}
\begin{proof}
This is an immediate consequence of Proposition \ref{prop:formalFM} and Theorem \ref{FML}. Using the notations of Theorem \ref{FML} we have that 
$$
\mathrm{FM}_E(\mathcal{O}_e) \simeq \mathcal{O}_E
$$
Since $\mathrm{FM}_E$ is an equivalence of $\mathbb{E}_\infty$-monoidal categories, it induces an equivalence of $\mathbb{E}_\infty$-algebras
$$
 \mathrm{Hom}_{\mathrm{QCoh}(E)^\star}(\mathcal{O}_e, \mathcal{O}_e) 
 \simeq  \mathrm{C}^* (E, \mathcal{O}) = \mathrm{Hom}_{\mathrm{QCoh}(E)^{\otimes}}(\mathcal{O}_E, \mathcal{O}_E).$$
Taking the cospectrum, and applying Proposition \ref{prop:formalFM}, we obtain the desired equivalence
$$
\mathrm{Aff}(E) =   \mathrm{coSpec} (\mathrm{Hom}_{\mathrm{QCoh}(E)^{\otimes}}(\mathcal{O}_E, \mathcal{O}_E)) \simeq 
\mathrm{coSpec} ( \mathrm{Hom}_{\mathrm{QCoh}(E)^\star}(\mathcal{O}_e, \mathcal{O}_e)) \simeq S^1_{\widehat{E}}
$$
\end{proof}
\begin{remark}
In fact, a stronger result than Corollary \ref{comparison} is true. Note that  
$\mathrm{Aff}(E)$ and $S^1_{\widehat{E}}$ are  naturally equipped with abelian group structures.   Corollary \ref{comparison} can be enhanced to an equivalence of abelian group objects in coaffine stacks.  %The equivalences of categories constructed in Proposition \ref{prop:formalFM} and Theorem \ref{FML} are  symmetric monoidal; but in fact, they preserve more structure. 
The categories appearing in Proposition \ref{prop:formalFM} and Theorem \ref{FML} are naturally commutative and cocommutative Hopf algebra objects in $\mathrm{Pr}^{\mathrm{L}}$, in the sense of \cite{ELL1}. The functors given in Proposition \ref{prop:formalFM} and Theorem \ref{FML}  are both equivalences of Hopf categories. In fact, in the case of Theorem \ref{FML}  this follows from  Proposition 5.1.3 \cite{ELL1}. Proposition \ref{prop:formalFM} can also be upgraded to an equivalence of Hopf categories. At the level of the endomorphisms of the unit objects, this encodes the extra data required to compare the abelian group structures on $\mathrm{Aff}(E)$ and $S^1_{\widehat{E}}$. 
However, since the  group structures do not enter the definition of $\mathbb{G}$-Hochschild homology, we do not pursue this story in any further detail here. We stress that the group structure should play a key role if one wishes to define   \emph{negative cyclic}   $\mathbb{G}$-Hochschild homology. This is an interesting problem, which we intend to return to in future work. 
\end{remark}

Corollary \ref{comparison} immediately implies that  two alternative  definitions of elliptic Hochschild homology which  given independently in \cite{MRT} and \cite{sibilla2023equivariant} in fact agree. To explain this point, let us recall briefly the setting of \cite{sibilla2023equivariant}. In \cite{sibilla2023equivariant} we define  more generally elliptic Hochschild homology of quotient stacks $S$ of the form $[X/G]$ where $X$ is a scheme and $G$ is a reductive algebraic group. Since our goal is comparing our construction to the one in \cite{MRT}, where only affine schemes are considered, we can focus on the more restrictive setting where $G$ is trivial. 

\begin{definition}
Let $E$ be an elliptic curve over $R$. Let $X$ be a scheme over $R$. 
\begin{itemize}
\item We denote by 
$ \, 
\mathrm{Map}_{\mathrm{dSt}_R}^0(E, X)
$ 
the smallest clopen substack  of $\mathrm{Map}_{\mathrm{dSt}_R}(E, X)$ containing the constant maps, i.e. the maps factoring as 
$ \, 
E \to \mathrm{Spec}(R) \to X
$. We call $ \mathrm{Map}_{\mathrm{dSt}_R}^0(E, X) $  the stack of almost constant maps from $E$ to $X$. 
\item The elliptic Hochschild homology of $X$ is defined as the global  sections of the structure sheaf of $ \mathrm{Map}_{\mathrm{dSt}_R}^0(E, X) $
$$
\mathrm{HH}_E(X) := \mathcal{O}(\mathrm{Map}_{\mathrm{dSt}_R}^0(E, X) ) 
$$
\end{itemize}
\end{definition}

\begin{corollary}
\label{compHH}
Let $E$ be an elliptic curve over $R$, and let $\widehat{E}$ be its completion at the identity section. Let $X = \mathrm{Spec}(A)$ be an affine scheme over $R$. Then there is a natural equivalence 
$$
\mathrm{HH}^{\widehat{E}}_*(X) \simeq \mathrm{HH}_E(X)
$$
\end{corollary} 
\begin{proof}
Since $X$ is affine,   the natural map 
$$
\mathrm{Map}_{\mathrm{dSt}_R}(\mathrm{Aff}(E), X) \to \mathrm{Map}_{\mathrm{dSt}_R}(E, X)
$$
is an equivalence. This implies in particular that there is an equivalence 
$$
t_0(\mathrm{Map}_{\mathrm{dSt}_R}(E, X)) \simeq t_0(X)
$$
As a consequence, there are identifications 
$$\mathrm{Map}_{\mathrm{dSt}_R}^0(E, X) = \mathrm{Map}_{\mathrm{dSt}_R}(E, X) \simeq \mathrm{Map}_{\mathrm{dSt}_R}(\mathrm{Aff}(E), X).$$
Combining this with Corollary \ref{comparison} we obtain the following chain of equivalences 
$$
\mathrm{HH}_E(X)   = \mathcal{O}(\mathrm{Map}_{\mathrm{dSt}_R}^0(E, X) )  \simeq 
\mathcal{O}(\mathrm{Map}_{\mathrm{dSt}_R}(\mathrm{Aff}(E), X) )
\simeq 
\mathcal{O}(\mathrm{Map}_{\mathrm{dSt}_R}(S^1_{\widehat{E}}, X) )
=  \mathrm{HH}^{\widehat{E}}_*(X)$$
and this concludes the proof.
\end{proof}
Next, we are  going to turn attention to the nodal and cuspidal cubic curves, which we will denote $\mathcal{N}$ and $\mathcal{C}$  respectively. %Over a field, $\mathcal{N}$ and $\mathcal{C}$ arise as (reduced and irreducible)  degenerations of smooth elliptic curves. However 
%The curves $\mathcal{N}$ and $\mathcal{C}$ can be defined over any ring $R$. 
Our next result is an analogue of Corollary \ref{comparison} for these  objects. In this setting, it is possible to argue in a different and  more direct  way than in Corollary \ref{comparison}. Even if the proof we shall give is quicker, it would be interesting to have also an alternative argument which is more directly akin to the smooth elliptic curve case.  We believe that the main lines of the argument leading up to Corollary \ref{comparison} could be adapted to the nodal and cuspidal cases. There is however an important technical input which would have to be established in that setting, namely an analogue of  Theorem \ref{FML}. We  comment on this point in some more detail in Remark \ref{altproof} below.

\begin{proposition}
\label{nodalcusp}
There are equivalences of coaffine stacks
$$
\mathrm{Aff}(\mathcal{N}) \simeq S^1_{\widehat{\mathbb{G}_m}} \simeq \mathrm{Aff}(S^1_{\mathrm{B}}), \quad \mathrm{Aff}(\mathcal{C}) \simeq S^1_{\widehat{\mathbb{G}_a}}
$$
\end{proposition}
\begin{proof}
The proofs in the two cases are very similar. Let us start from the nodal curve $\mathcal{N}$. The key observation is that $\mathcal{N}$ can be realized as the following push-out in schemes 
\begin{equation}
\label{nodal}
\begin{gathered}
\xymatrix{\Spec(R) \coprod \Spec(R) \ar[r]^-{\iota} \ar[d]& \mathbb{P}^1_R\ar[d]\\
\Spec(R)  \ar[r] & \mathcal{N}}
\end{gathered}
\end{equation}
where $\iota$ denotes the inclusion of two points in $\mathbb{P}^1_R$ as $0$ and $\infty$. This presentation  yields a map $c: \mathcal{N} \to S^1_{\mathrm{B}}$ via the following morphism of push-out diagrams
$$
\xymatrix{ \\
\Spec(R) \coprod \Spec(R) \ar[r]^-{\iota}  \ar@/^2pc/[rrr] \ar[d]& \mathbb{P}^1_R\ar[d] \ar@/^2pc/[rrr] & &   \Spec(R) \coprod \Spec(R) \ar[r] \ar[d] & \Spec(R) \ar[d] \\ 
\Spec(R)  \ar[r] \ar@/^2pc/[rrr] & \mathcal{N} 
\ar@{-->}@/^2pc/[rrr] && \Spec(R) \ar[r] &S^1_{\mathrm{B}}  
}
$$
where the curved arrows joining the two diagrams are the obvious ones. The claim is that $c: \mathcal{N} \to S^1_{\mathrm{B}}$ induces an equivalence between the affinizations. By the universal property of the affinization, this can be tested by showing that for any affine scheme $T$, the morphism 
$$
\mathcal{L} T = \mathrm{Map}_{\mathrm{dSt}_R}(S^1_{\mathrm{B}}, T) \stackrel{c^*}\to \mathrm{Map}_{\mathrm{dSt}_R}(\mathcal{N}, T)
$$
is an equivalence.

Since (\ref{nodal}) is a push-out in schemes, we obtain a pull-back square 
\[
\xymatrix{
\mathrm{Map}_{\mathrm{dSt}_R}(\mathcal{N}, T) \ar[r] \ar[d]& \mathrm{Map}_{\mathrm{dSt}_R}(\mathbb{P}^1, T) \ar[d]\\
T\ar[r]& T \times T }
\] 
Now, since $T$ is affine, we have that
$$
\mathrm{Map}_{\mathrm{dSt}_R}(\mathbb{P}^1, T) \simeq 
\mathrm{Map}_{\mathrm{dSt}_R}(\mathrm{Aff}(\mathbb{P}^1), T).
$$
Further, since $\mathcal{O}(\mathbb{P}^1_R) \simeq R$, it follows that 
$ 
\mathrm{Aff}(\mathbb{P}^1) \simeq \mathrm{Spec}(R).
$ 
As a result, we obtain equivalences 
$$\mathrm{Map}_{\mathrm{dSt}_R} (\mathcal{N}, T)  \simeq \mathcal{L}T \simeq \mathrm{Map}_{\mathrm{dSt}_R} (S^1_{\mathrm{B}}, T).$$
As a consequence, as we claimed, we have an equivalence $\Aff{\mathcal{N}}\simeq \Aff{S^1}$. 

The case of the cuspidal curve $\mathcal{C}$ is entirely analogous, so we shall give a more succinct treatment of the argument. 
In the category of schemes, we can realise the the cusp  as the pushout

\[\xymatrix{\Spec(R[T]/T^2) \ar[r] \ar[d]& \mathbb{P}^1\ar[d]\\
\Spec(R) \ar[r]&  \mathcal{C}}\]
Arguing as before, it is sufficient to show that for any affine scheme $T$ there is an equivalence 
$$
\mathrm{Map}_{\mathrm{dSt}_R}(\mathcal{C}, T) \simeq \mathbb{T}[-1]T
$$ 
where $\mathbb{T}[-1]T$ is the total space of the shifted tangent bundle of $T$. Recall that there is an equivalence 
$$
\mathrm{Map}_{\mathrm{dSt}_R}(\Spec(R[T]/T^2, T) \simeq \mathbb{T}T
$$
Then, as in the nodal case, the claim follows immediately from the fact that there is a pull-back diagram 
\[\xymatrix{\mathrm{Map}_{\mathrm{dSt}_R}(\mathcal{C}, T) \ar[r] \ar[d]&Y\ar[d]\\
Y\ar[r]& \mathbb{T}Y }\]
\end{proof}

%Before going further, we would like to  comment on  Proposition \ref{nodalcusp}. 
\begin{remark}
\label{altproof}
We believe that the proof strategy adopted in the elliptic case could be pushed to cover also the setting of the nodal and cuspidal curves. In order to implement that argument however one would have to establish some preliminary results on Fourier--Mukai equivalences which as far as we know have not been established in the literature even in the classical setting $R=\mathbb{C}$.

Let us explain the story in the case of the nodal curve $\mathcal{N}$, as the situation for $\mathcal{C}$ is entirely analogous. Let $U$ be the complement of the node, and note that the choice of a point  $e$ in $U$ determines an isomorphism
$ \, 
U \cong \mathbb{G}_m
$. Let $i: \mathbb{G}_m \to \mathcal{N}$ the corresponding open embedding. 
 We believe that there should be a Fourier--Mukai equivalence  
$$
\mathrm{FM}_\mathcal{N}: \mathrm{QCoh}(\mathcal{N})  \to \mathrm{QCoh}(\mathcal{N})
$$
 with the following properties.  Let  $\star$ be the tensor product on $\mathrm{QCoh}(\mathcal{N})$ obtained by pulling back the ordinary tensor product along $\mathrm{FM}_\mathcal{N}$. Then:
 \begin{itemize}
 \item the skyscraper sheaf $\mathcal{O}_e$ is the unit for the $\star$-tensor product
 \item the push-forward  functor 
 $$
 i_*: \mathrm{QCoh}(\mathbb{G}_m)^{\otimes^{\mathrm{c}}}  \to \mathrm{QCoh}(\mathcal{N})^\star 
 $$ 
is symmetric monoidal, where $\otimes^{\mathrm{c}}$ denotes the convolution tensor product $\mathrm{QCoh}(\mathbb{G}_m)$
\end{itemize}
If this was the case, one could replicate the strategy we used to prove \ref{comparison} in order to show that 
$$
\mathrm{Aff}(\mathcal{N}) \simeq S^1_{\widehat{\mathbb{G}_m}}
$$
We stress that in fact it is not difficult to define a candidate for $\mathrm{FM}_\mathcal{N}$. Over a field, the ordinary Fourier--Mukai equivalence for elliptic curves 
$\mathrm{FM}_E$ can be factored as a composition of spherical functors. These spherical twists make  sense also in the setting of $\mathcal{N}$. So there is a very natural a definition of $\mathrm{FM}_\mathcal{N}$ and  it remains to check whether it has the desired properties.  %Although there are several technical problems to circumvent when working over a general base ring $R$ as opposed to a field, but they should not be major. 
We believe that this circle of ideas merits further investigation, which we leave to future work.
\end{remark}

\section{TMF Hochschild homology}
%\subsection{Recollections on moduli of elliptic curves}
\label{section:Ell}
%In this section we use  Corollary  \ref{compHH} and Proposition \ref{nodalcusp} to define global versions of $\bG$-Hochschild homology over moduli spaces of elliptic curves and cubics. 
We define next three universal versions of $\widehat{E}$-Hochschild homology in the spirit of the theory of topological modular forms. This yields global versions of the $\widehat{E}$-Hochschild homology from \cite{MRT} over $\cM_{ell}$, $\overline{\cM_{ell}}$ and $\cM_{cub}$. The key inputs are Corollary  \ref{compHH} and Proposition \ref{nodalcusp}, which allow us to model $\mathbb{G}$-Hochschild homology, when $\mathbb{G}$ is the completion of an algebraic group scheme, as functions on mapping stacks from algebraic curves. As we explained the approach of \cite{sibilla2023equivariant}, which we adopt here, has the advantage of giving a meaningful  theory also beyond the affine setting. Here, however,  we shall limit ourselves to consider affine schemes over $R$, as this is the setting originally considered in \cite{MRT}. Extensions of this picture to more general objects, and in particular quotient stacks, will appear in \cite{ChristianPaolo}.

We start by recalling a few generalities regarding the moduli theory of elliptic curves and cubics. A reference is \cite[Section 2.4]{Meierthesis} in the smooth situation, and \cite[Section 3]{Mathewtmf} in the general case. In this paper, we do not consider the spectral enhancements developed by Lurie in his series of papers \cite{ELL1}, \cite{ELL2} and \cite{ELL3}.
\begin{definition}
An elliptic curve over a scheme $B$ is a proper smooth morphism $\pi: E\to B$ equipped with a section $e:B\to E$ such that, for every geometric point $b:\Spec k\to B$, the fiber $E\times_B b$ is a connected curve of genus 1.
\end{definition}
Consider the functor 
$$F_{\cM_{ell}}:\CAlg^\heartsuit\to\cS$$
assigning to a commutative ring $R$ the groupoid of elliptic curves over $\Spec R$ and   isomorphisms between them. This functor is represented by a classical Deligne--Mumford stack $\cM_{ell}$. There are some important variants of this basic moduli stack. The first is the modular compactification of $\cM_{ell}$, $\overline{\cM_{ell}}$. This is a Deligne--Mumford stack proper over $\Spec \bZ$, which classifies proper, flat morphisms of finite presentation over $R$ whose geometric fibers are curves of arithmetic genus 1 which are either smooth or have one nodal singularity, and are equipped with a marked point that lies in their smooth locus \cite[Definition 2.6]{Mathewtmf}. The other moduli stack we will  consider is a version of $\overline{\cM_{ell}}$ which  additionally parametrizes the cuspidal cubic. This is the moduli stack of cubics considered in \cite[Section 3]{Mathewtmf}. We recall the definition below.
\begin{definition}
A cubic curve over a scheme $B$ is a proper flat morphism of finite presentation $\pi: E\to B$ together with a section $e:B\to E$ whose image is contained in the smooth locus, having the property that for every geometric point $b:\Spec k\to B$, the fiber $E\times_B b$ is a reduced irreducible curve of arithmetic genus 1. 
\end{definition}
Let $\cM_{cub}$ denote the algebraic stack representing the functor sending a commutative ring $R$ to the space of cubic curves over $\Spec R$ and their isomorphisms. We remark that $\cM_{cub}$   is not a Deligne--Mumford stack. We denote respectively 
$$
\cE^{un}, \quad \overline{\cE}^{un} \quad \text{and} \quad \cC^{un}
$$
the universal curves on $\cM_{ell}$, $\overline{\cM_{ell}}$ and $\cM_{cub}$.

Before proceeding, we need to fix some notations. Let $\mathcal{X}$ be a  derived stack over $R$.  We denote by $\mathrm{CAlg}_{\mathcal{X}}^{\mathrm{cn}}$ the $\infty$-category of connective commutative algebra objects in $\mathrm{QCoh}(\mathcal{X})$. If $\cX$ is equivalent to its truncation, i.e. it is an object in $\mathrm{St}_R$, we can define  coaffine stacks over $\cX$, and the affinization functor,  exactly as in the absolute case: if $Y \to \cX$ is in $\mathrm{St}_R/\cX$ we denote its affinization $\mathrm{Aff}_{\mathcal{X}}(Y)$. 
We denote by $\mathrm{Map}_{\mathrm{dSt}_R/\mathcal{X}}(-,-)$  the mapping stack in the $\infty$-category of stacks over $\mathcal{X}$. Let $Y_1 \to \mathcal{X}$ and $Y_2 \to \mathcal{X}$ be  stacks over $\mathcal{X}$. We define the stack of almost constant maps 
 as the connected component of the sections of $Y_2 \to \mathcal{X}$, and we let  
 $$
\pi: \mathrm{Map}_{\mathrm{dSt}_R/\mathcal{X}}^0(Y_1, Y_2) \to \mathcal{X}
 $$
be the structure morphism. An example which will be important for our next definition is obtained by setting $\mathcal{X}=\cM_{ell}$, $Y_1=\cE^{un}$ and $Y_2 = X \times \cM_{ell}$ where $X$ is a scheme. This gives the stack of almost constant maps from the universal elliptic curve to $X$
$$
\pi: \mathrm{Map}_{\mathrm{dSt}_R/\cM_{ell}}^0(\cE^{un}, X \times \cM_{ell}) \to \cM_{ell}
$$
 %This object, in the generalised elliptic curves case, should be thought as analogous to the filtered Hochschild homology as in providing a degeneration of certain variants of Hochschild homology to the more traditional one.
%Let $\cM_{ell}$ be the classical Deligne--Mumford moduli stack of elliptic curves as recalled in Section \ref{section:Ell}, and let $\cE^{un}$ be the universal elliptic curve over $\cM_{ell}$. 

\begin{definition}
\label{TMFTMF}
We define sheafified $\mathrm{TMF}$-Hochschild homology to be the functor
$$
\mathcal{HH}^{\mathrm{TMF}}_* : \mathrm{CAlg}_{R}^{\mathrm{cn}} \to \mathrm{CAlg}_{  \cM_{ell} }^{\mathrm{cn}}  \, ,   \quad A \mapsto \mathcal{HH}^{\mathrm{TMF}}_*(A):=\pi_*(\mathcal{O}_{\mathrm{Map}_{\mathrm{dSt}_R/\cM_{ell}}^0(\cE^{un} \, , \,   \mathrm{Spec}(A) \times \cM_{ell})})
$$ 
We denote by $\mathrm{HH}^{\mathrm{TMF}}_*(A)$ the global sections of  $\mathcal{HH}^{\mathrm{TMF}}_*(A)$. We define $\mathrm{TMF}$-Hochschild homology to be the functor
$$
\mathrm{HH}^{\mathrm{TMF}}_*: \mathrm{CAlg}_{R}^{\mathrm{cn}} \to \mathrm{CAlg}_{R}   \, ,   \quad A \mapsto \mathrm{HH}^{\mathrm{TMF}}_*(A) 
$$ 
%The $\mathrm{Tmf}$-Hochschild homology of $X$ is the derived global functions of the derived mapping stack
%$\Mapo{\overline{\cE}^{un}}{X\times\overline{\cM_{ell}}}_{/\overline{\cM_{ell}}}$, i.e.
%$$\mHH_{\mathrm{Tmf}}(X):=\cO(\Mapo{\overline{\cE}^{un}}{X\times\overline{\cM_{ell}}}_{/\overline{\cM_{ell}}})$$
%The $\mathrm{tmf}$-Hochschild homology of $X$ is the derived global functions of the derived mapping stack
%$\Mapo{\cC^{un}}{X\times\cM_{cub}}_{/\cM_{cub}}$, i.e.
%$$\mHH_{\mathrm{tmf}}(X):=\cO(\Mapo{\cC^{un}}{X\times\cM_{cub}}_{/\cM_{cub}})$$
\end{definition}

Note that since $\mathrm{Spec}(A) \times \cM_{ell}$ is affine over 
$\cM_{ell}$ there is an equivalence 
$$\mathrm{Map}_{\mathrm{dSt}_R/\cM_{ell}}^0(\cE^{un}, \mathrm{Spec}(A) \times \cM_{ell}) \simeq \mathrm{Map}_{\mathrm{dSt}_R/\cM_{ell}}^0(\mathrm{Aff}_{\cM_{ell}}(\cE^{un}), \mathrm{Spec}(A) \times \cM_{ell}) $$
where the morphism is induced by the natural map $\cE^{un}\to \mathrm{Aff}_{\cM_{ell}}(\cE^{un})$.  
Further, the stack of almost constant maps out of $\mathrm{Aff}_{\cM_{ell}}(\cE^{un})$ actually coincides with the full mapping stack
$$
\mathrm{Map}_{\mathrm{dSt}_R/\cM_{ell}}(\mathrm{Aff}_{\cM_{ell}}(\cE^{un}), \mathrm{Spec}(A) \times \cM_{ell})
$$
Thus restricting to affine schemes, as we do throughout the article, simplifies somewhat the    geometry  underpinning $\mathcal{HH}^{\mathrm{TMF}}_*$.

The following result shows that $\cHH_\ast^{\mathrm{TMF}}$   interpolates between different $\widehat{E}$-Hochschild homologies, where $\widehat{E}$ is the completion at the identity of an elliptic curve $E$, as $E$ varies in the moduli of elliptic curves. 
 We view therefore  $\cHH_\ast^{\mathrm{TMF}}$ as conceptually related to the filtered Hochschild homology considered in 
\cite{MRT} and \cite{moulinos2024filtered}, which is a global object over $[\mathbb{A}^1/\mathbb{G}_m]$  encoding a deformation  of $\widehat{G}$-Hochschild homology to Hodge Hochschild homology. In fact, we will be able to draw a closer connection between our theory and the picture in \cite{MRT} and \cite{moulinos2024filtered} below: this requires introducing a variant of TMF Hochschild homology which lives over the moduli of cubics.
\begin{proposition}\label{prop:fibersell}
Let $E:\Spec R\to\cM_{ell}$ be an elliptic curve and $X=\Spec(A)$ be an affine scheme over $\Spec R$. Then the fiber of $\cHH_\ast^{\mathrm{TMF}}(X)$ over $E$ is equivalent to $\mHH_\ast^{\widehat{E}}(X)$. In particular, we obtain a restriction map  
$$\mHH^{\mathrm{TMF}}_\ast(X)\to\mHH_\ast^{\widehat{E}}(X)$$
\end{proposition}
\begin{proof}
The statement follows from base-change along the following pullback diagram:
$$
\xymatrix{
\mathrm{Map}_{\mathrm{dSt}_R}(\mathrm{Aff}(E),X) \ar[r]^-{E'}\ar[d]^{\pi'} & \mathrm{Map}_{\mathrm{dSt}_R/\cM_{ell}}(\mathrm{Aff}_{\cM_{ell}}(\cE^{un}), X \times \cM_{ell}) \ar[d]^{\pi} \\
\Spec R\ar[r]^{E} & \cM_{ell}
}
$$
Indeed, if we apply base change to the structure sheaf of $\mathrm{Map}_{\mathrm{dSt}_R/\cM_{ell}}(\mathrm{Aff}_{\cM_{ell}}(\cE^{un}), X \times \cM_{ell})$ we obtain an equivalence
$$
E^*(\cHH_\ast^{\mathrm{TMF}}(X)) \simeq \mathcal{O}(\mathrm{Map}_{\mathrm{dSt}_R}(\mathrm{Aff}(E),X))$$
and the latter coincides with $\mHH_\ast^{\widehat{E}}(X)$ 
 by Corollary 
\ref{compHH}.

In order to prove that the Beck--Chevalley maps 
$$E^\ast\pi_\ast\to(\pi')_\ast (E')^\ast$$
are isomorphisms, we want to apply \cite[Proposition 3.10]{ben2010integral} to the map $\pi$. To do so, we need to show $\pi$ is a perfect map; we will in fact prove the stronger statement that $\pi$ is affine. By the pullback diagram above, it is sufficient to observe that the derived stack
$$\mathrm{Map}_{\mathrm{dSt}_R}(\mathrm{Aff}(E),X)$$
is in fact a derived affine scheme. The proof is essentially identical to that of \cite[Proposition 5.1.2]{MRT}. First of all, its truncation is $X$, hence a derived affine scheme. Moreover, the necessary extra conditions required by Artin--Lurie representability (i.e. having a global cotengent complex, admitting an obstruction theory and being nil-complete) are satisfied by mapping stacks with sources classifying stacks of group schemes. By Corollary \ref{comparison} we have that $\Aff{E}\simeq S^1_{\widehat{E}}$ and so is precisely of that form.
%Let $E:\Spec R\to\cM_{ell}$ be an elliptic curve over $\Spec R$. The definition immediately implies the following diagram is a pullback:
%$$
%\xymatrix{
%S^1_{\widehat{E}}\ar[r]\ar[d] & \mathrm{Aff}_{/\cM_{ell}}(\cE^{un})\ar[d] \\
%\Spec R \ar[r]^{E} & \cM_{ell}
%}
%$$
%\begin{proof}
%This is a consequence of the definition of $S^1_{\mathrm{TMF}}$. In fact, by Remark \ref{remark:relativeaff}, we have that 
%$$S^1_{\mathrm{TMF}}=\Aff{\cE^{un}}$$
%and by definition, the pullback of the relative affinization is
%$$S^1_{\mathrm{TMF}}\times_{\cM_{ell}}\Spec R=\Aff{\cE^{un}}\times_{\cM_{ell}}\Spec R=\Aff{E}=S^1_{\widehat{E}}$$ 
%by \ref{comparison}.
%\end{proof}
%In particular, this implies that for any affine scheme $X$ over $\Spec R$ and any elliptic curve $E$ over $\Spec R$, we have a map
%$$\mathrm{Map}^0_{/\Spec R}(E, X)\to \mathrm{Map}^0_{/\cM_{ell}}(\cE^{un},X\times\cM_{ell})$$
\end{proof}

\begin{remark}\label{remark:relativeaff}
In the spirit of \cite{MRT}, we could define the $\mathrm{TMF}$-circle as 
$$S^1_{\mathrm{TMF}}:=B_{/\cM_{ell}}(\widehat{\cE^{un}})^\vee$$
where $(\widehat{\cE^{un}})^\vee$ denotes the  relative Cartier dual  of the formal completion of the universal elliptic curve at its unit section, and $B_{/\cM_{ell}}(-)$ is the relative classifying stack. Here we are taking for granted that setting  up a theory of Cartier duality over a base stack, and in particular over $\cM_{ell}$, is indeed feasible; but of course this would have to be worked out carefully. If this works, we should expect that an analogue of  
%Note that $S^1_{\mathrm{TMF}}$ is a coaffine stack over  $\cM_{ell}$.
 Corollary \ref{comparison} would hold in this setting: i.e. we should have  an equivalence 
 $$S^1_{\mathrm{TMF}} \simeq \mathrm{Aff}_{\cM_{ell}}(\cE^{un})$$ 
 Then we could give an alternative but equivalent definition of 
 $\cHH_\ast^{\mathrm{TMF}}$ in terms of $S^1_{\mathrm{TMF}}$. Namely, in keeping with the approach in  \cite{MRT}, if $X$ is an affine scheme we could set 
  $$\cHH_\ast^{\mathrm{TMF}}(X):=\cO(\mathrm{Map}_{\mathrm{dSt}_R/\cM_{ell}}(S^1_{\mathrm{TMF}},X\times\cM_{ell})).$$
  As already discussed however, extending the theory to non-affine setting would require taking the approach of Definition \ref{TMFTMF}.\end{remark}
 
%\begin{definition}
%Let $X=\Spec A$ be an affine scheme. The $\mathrm{Tmf}$-Hochschild homology of $X$ is the derived global functions of the derived mapping stack
%$\Mapo{\overline{\cE}^{un}}{X\times\overline{\cM_{ell}}}_{/\overline{\cM_{ell}}}$, i.e.
%$$\mHH_{\mathrm{Tmf}}(X):=\cO(\Mapo{\overline{\cE}^{un}}{X\times\overline{\cM_{ell}}}_{/\overline{\cM_{ell}}})$$
%The $\mathrm{tmf}$-Hochschild homology of $X$ is the derived global functions of the derived mapping stack
%$\Mapo{\cC^{un}}{X\times\cM_{cub}}_{/\cM_{cub}}$, i.e.
%$$\mHH_{\mathrm{tmf}}(X):=\cO(\Mapo{\cC^{un}}{X\times\cM_{cub}}_{/\cM_{cub}})$$
%\end{definition}
%The following might be true and maybe worth conjecturing:
%We have very similar statements as in the $\mathrm{TMF}$ case. We can define
%$$S^1_{\mathrm{Tmf}}:=B_{/\overline{\cM_{ell}}}\widehat{\overline{\cE}^{un}}^\vee$$
%$$S^1_{\mathrm{tmf}}:=B_{/\cM_{cub}}\widehat{\cC^{un}}^\vee$$
%and we have \textcolor{red}{only in char 0 or p}
%$$S^1_{\mathrm{Tmf}}=\Aff{\overline{\cE}^{un}}$$
%$$S^1_{\mathrm{tmf}}=\Aff{\cC^{un}}$$
%and
%$$\mHH_{\mathrm{Tmf}}(X):=\cO(\Map{S^1_{\mathrm{Tmf}}}{X\times\overline{\cM_{ell}}}_{/\overline{\cM_{ell}}})$$
%$$\mHH_{\mathrm{tmf}}(X):=\cO(\Map{S^1_{\mathrm{tmf}}}{X\times\cM_{cub}}_{/\cM_{cub}})$$

We  conclude this section by extending our constructions to the setting of the modular compactification $\overline{\cM_{ell}}$ of $\cM_{ell}$ and to the moduli stack of cubics $\cM_{cub}$.  
\begin{definition}
We define sheafified $\mathrm{Tmf}$-Hochschild homology to be the functor
$$
\mathcal{HH}^{\mathrm{Tmf}}_* : \mathrm{CAlg}_{R}^{\mathrm{cn}} \to \mathrm{CAlg}_{  \overline{\cM_{ell}} }^{\mathrm{cn}}  \, ,   \quad A \mapsto \mathcal{HH}^{\mathrm{Tmf}}_*(A):=\pi_*(\mathcal{O}_{\mathrm{Map}_{\mathrm{dSt}_R/\overline{\cM_{ell}}}^0(\overline{\cE}^{un} \, , \, \mathrm{Spec}(A) \times \overline{\cM_{ell}})})
$$ 
We denote by $\mathrm{HH}^{\mathrm{Tmf}}_*(A)$ the global sections of $\mathcal{HH}^{\mathrm{Tmf}}_*(A)$. We define $\mathrm{Tmf}$-Hochschild homology to be the functor
$$
\mathrm{HH}^{\mathrm{Tmf}}_*: \mathrm{CAlg}_{R}^{\mathrm{cn}} \to \mathrm{CAlg}_{R}   \, ,   \quad A \mapsto \mathrm{HH}^{\mathrm{Tmf}}_*(A) 
$$ 
\end{definition}
Similarly to the Tmf-case, we can formulate an analogous definition over the moduli stack of cubics.
\begin{definition}
We define sheafified $\mathrm{tmf}$-Hochschild homology to be the functor
$$
\mathcal{HH}^{\mathrm{tmf}}_* : \mathrm{CAlg}_{R}^{\mathrm{cn}} \to \mathrm{CAlg}_{  \cM_{cub} }^{\mathrm{cn}}  \, ,   \quad A \mapsto \mathcal{HH}^{\mathrm{tmf}}_*(A):=\pi_*(\mathcal{O}_{\mathrm{Map}_{\mathrm{dSt}_R/\cM_{cub}}^0(\cC^{un} \, , \, \mathrm{Spec}(A) \times \cM_{cub})})
$$ 
We denote by $\mathrm{HH}^{\mathrm{tmf}}_*(A)$ the global sections of $\mathcal{HH}^{\mathrm{tmf}}_*(A)$. We define $\mathrm{tmf}$-Hochschild homology to be the functor
$$
\mathrm{HH}^{\mathrm{tmf}}_*: \mathrm{CAlg}_{R}^{\mathrm{cn}} \to \mathrm{CAlg}_{R}   \, ,   \quad A \mapsto \mathrm{HH}^{\mathrm{tmf}}_*(A) 
$$ 
\end{definition}

We would like to have a similar description of the fibers of $\cHH_\ast^{\mathrm{tmf}}$ as the one Proposition \ref{prop:fibersell} gives over the moduli stack of elliptic curves. In order to achieve this goal, we focus first on the case of curves over algebraically closed fields $k$. We do so because there is a complete classification of cubic curves over algebraically closed fields: these can only be elliptic curves, or curves with a single nodal singularity or a single cuspidal singularity. Moreover, the formal completion at the unit section is always one of the possible types: it is the formal additive group $\widehat{\bG}_a$ for the cuspidal curve, the formal multiplicative group $\widehat{\bG}_m$ for the nodal curve, and in the smooth case it is an elliptic formal group. In this case, the same arguments of Proposition \ref{prop:fibersell} prove the following

\begin{proposition}\label{prop:algclosed}
Let $k$ be an algebraically closed field and $X$ be an affine scheme over $k$. We have the following:
\begin{enumerate}
\item
Let $C:\Spec k\to\overline{\cM_{ell}}$ be a generalized elliptic curve over $k$.
The fiber of $\cHH_\ast^{\mathrm{Tmf}}(X)$ over $C$ is 
\begin{itemize}
\item $\mHH_\ast^{\widehat{\bG}_m}(X)$ if $C$ is a nodal cubic
\item $\mHH_\ast^{\widehat{E}}(X)$ if $C=E$ is an elliptic curve.
\end{itemize}
\item
Let $C:\Spec k\to\cM_{cub}$ be a cubic curve over $k$.
The fiber of $\cHH_\ast^{\mathrm{tmf}}(X)$ over $C$ is 
\begin{itemize}
\item $\mHH_\ast^{\widehat{\bG}_a}(X)$ if $C$ is a cuspidal cubic
\item $\mHH_\ast^{\widehat{\bG}_m}(X)$ if $C$ is a nodal cubic
\item $\mHH_\ast^{\widehat{E}}(X)$ if $C=E$ is an elliptic curve.
\end{itemize}
\end{enumerate}
\end{proposition}

The case of general base rings $R$ is more complicated, as there is no trichotomy in the classification of cubic curves. Nevertheless, we have the following statement, proved by the same arguments of Proposition \ref{prop:fibersell}:\begin{proposition}
Let $R$ be a ring and $X$ be an affine scheme over $R$. We have the following:
\begin{enumerate}
\item
Let $C:\Spec R\to\overline{\cM_{ell}}$ be the either nodal curve over $R$ or an elliptic curve over $R$.
The fiber of $\cHH_\ast^{\mathrm{Tmf}}(X)$ over $C$ is 
\begin{itemize}
\item $\mHH_\ast^{\widehat{\bG}_m}(X)$ if $C=\cN$ is the nodal curve
\item $\mHH_\ast^{\widehat{E}}(X)$ if $C=E$ is an elliptic curve.
\end{itemize}
\item
Let $C:\Spec R\to\cM_{cub}$ be either the nodal or the cuspidal curve over $R$, or an elliptic curve over $R$.
The fiber of $\cHH_\ast^{\mathrm{tmf}}(X)$ over $C$ is 
\begin{itemize}
\item $\mHH_\ast^{\widehat{\bG}_a}(X)$ if $C=\cC$ is a the cuspidal curve
\item $\mHH_\ast^{\widehat{\bG}_m}(X)$ if $C=\cN$ is the nodal curve
\item $\mHH_\ast^{\widehat{E}}(X)$ if $C=E$ is an elliptic curve.
\end{itemize}
\end{enumerate}
\end{proposition}
%
%These two variants allow us some extra degenerations compared to the TMF version described in the previous paragraph. Indeed, the presence of the nodal and cuspidal curve allows to consider degenerations to Hochschild homology and its additive conterpart, Hodge cohomology. By the discussion in the previous paragraph combined with the results on the nodal and cuspidal curves in section ... , we have maps
%$$\mHH_{\mathrm{Tmf}}(X)\to\mHH_{\widehat{\bG}_{m,R}}(X)$$
%if $\Spec R\to\overline{\cM_{ell}}$ is the map classifying the nodal curve over $R$, and 
%$$\mHH_{\mathrm{tmf}}(X)\to\mHH_{\widehat{\bG}_{m,R}}(X)$$
%$$\mHH_{\mathrm{tmf}}(X)\to\mHH_{\widehat{\bG}_{a,R}}(X)$$
%for $\Spec R\to\cM_{cub}$ classifying respectively the nodal and cuspidal curves over $R$. 
For more general cubics $C:\Spec R\to\cM_{cub}$, the situation is more complicated. Indeed, the fiber of $\cHH_\ast^{\mathrm{tmf}}$ over $C$ is given by
$$\cO(\mathrm{Map_{/\Spec R}}(\Aff{C},X))$$
and in particular we have a restriction map
$$\mHH_\ast^{\mathrm{tmf}}(X):=\cO(\mathrm{Map}_{/\cM_{cub}}^0(\cC^{un},X\times\cM_{cub}))\to\cO(\mathrm{Map_{/\Spec R}}(\Aff{C},X))$$
%$$\to\cO(\mathrm{Map}_{/\Spec k}(\Aff{C_k},X\times_{\Spec R}\Spec k))=:\mHH_{\widehat{C}_{k}}(X)$$
Note that, as remarked in \cite{Mathewtmf}, the formal completion of any cubic curve $C:\Spec R\to\cM_{cub}$ is a formal group over $R$; let us denote it $\widehat{C}$.  In particular, it is possible to consider the associated Hochschild homology theory for affine schemes $X$ over $R$
$$\mHH_\ast^{\widehat{C}}(X):=\cO(\mathrm{Map}_{/\Spec R}(B\widehat{C}^\vee,X))$$
in the sense of \cite{MRT}.
We strongly believe there to be an equivalence
$$\mHH_\ast^{\widehat{C}}(X)\simeq\cO(\mathrm{Map_{/\Spec R}}(\Aff{C},X))$$
induced by an equivalence $S^1_{\widehat{C}}:=B\widehat{C}^\vee\simeq\Aff{C}$, but our techniques do not generalise immediately to the non-smooth case. Indeed, there is no Fourier--Mukai duality for cubics over general base rings, as opposed to the situation of elliptic curves. The existence of such a theory is an interesting question in its own right, and nontrivial even over algebraically closed fields of characteristic zero. We leave it as a conjecture, which is a generalization of the discussion in Remark \ref{altproof}:
\begin{conjecture}
Let $C:\Spec R\to \cM_{cub}$ be a cubic curve. There is a Fourier--Mukai symmetric monoidal equivalence
$$\QCoh(C)^\otimes\simeq\QCoh(C)^\star$$
where $\star$ denotes an exotic symmetric monoidal structure on $\QCoh(C)$ extending the convolution tensor product on quasi-coherent sheaves on the smooth locus of $C$. In particular, there is an isomorphism
$$\Aff{C}\simeq S^1_{\widehat{C}}$$
\end{conjecture}

%from the global sections to the Hochschild homology theories associated to the completion of $C_k$. This comparison map factors through the global sections of the mapping stack from $\Aff{C}$, which we would like to interpret as the Hochschild homology associated to the formal group given by $\widehat{C}$. To do so, we would need to identify $\Aff{C}$ with the $\widehat{C}$-twisted circle $S^1_{\widehat{C}}$ over any ring $R$, which our techniques do not allow if $C$ is not either smooth or a nodal or cuspidal cubic over $R$. Nevertheless, we believe this is true (a similar discussion appears in Remark \ref{altproof}). 
%Since, over general rings $R$, all three types can appear simultaneously as geometric fibers of a single cubic $C$, over each $C$ a similar degeneration happens as over the universal cubic $\cC^{un}$.

One of the main advantages of tmf-Hochschild homology over its counterparts Tmf- and TMF-Hochschild homology is the fact that it encodes degenerations both to standard Hochschild homology and to Hodge Hochschild homology. This aspect is shared with the filtered Hochschild homology of \cite{MRT}.
In this light, we view the two theories $\cHH_\ast^{\mathrm{tmf}}$ and $\mHH_\ast^{\mathrm{fil}}$ as parallel. 
We believe the parallelism between $\cHH_\ast^{\mathrm{tmf}}$ and $\mHH_\ast^{\mathrm{fil}}$ can be explained in terms of the degeneration of the nodal curve to the cuspidal curve, which can be viewed as a map
$$[\bA^1/\bG_{m}]\to \cM_{cub}$$
The pullback of the universal cubic curve over $[\bA^1/\bG_{m,}]$ is by definition a filtered stack: we call it the \emph{filtered nodal curve} $\cN_{\mathrm{fil}}$. We strongly believe that the filtered Hochschild homology of \cite{MRT} can be obtained as 
$$\mHH_\ast^{\mathrm{fil}}(X)=\cO(\mathrm{Map}_{[\bA^1/\bG_{m}]}(\mathrm{Aff}_{[\bA^1/\bG_{m}]}(\cN_{\mathrm{fil}}),X\times [\bA^1/\bG_{m}]))$$ 
We formulate the following
%We strongly believe that the pullback of the tmf-circle $\mathrm{Aff}_{\cM_{cub}}(\cC^{un})$ over $[\bA^1_k/\bG_{m,k}]$, or equivalently the affinization of the filtered nodal curve over $[\bA^1_k/\bG_{m,k}]$, is in fact the filtered circle. Moreover, we believe the fiber of $\cHH_\ast^{\mathrm{tmf}}(X)$ over the filtered nodal curve is the filtered Hochschild homology of \cite{MRT}. For now, we limit ourselves to formulating the
\begin{conjecture}
Let $R$ be a commutative ring, and 
$$[\bA^1/\bG_{m}]\to \cM_{cub}$$
be the map classifying the filtered nodal curve over $R$. Then its relative affinization is the filtered circle of \cite{MRT}
$$\mathrm{Aff}_{[\bA^1/\bG_{m}]}(\cN_{\mathrm{fil}})\simeq S^1_{\mathrm{fil}}$$
In particular, for an affine scheme $X$ over $R$, this induces an equivalence
$$\mHH_\ast^{\mathrm{fil}}(X)\simeq\cO(\mathrm{Map}_{[\bA^1/\bG_{m}]}(\mathrm{Aff}_{[\bA^1/\bG_{m}]}(\cN_{\mathrm{fil}}),X\times [\bA^1/\bG_m]))$$
\end{conjecture}
Work in progress of the third author with Christian Forero Pulido \cite{ChristianPaolo} will investigate this conjecture further.
%Then the following is a pullback diagram:
%$$
%\xymatrix{
%S^1_{fil}\ar[r]\ar[d] & S^1_{\mathrm{tmf}}\ar[d] \\
%[\bA^1_k/\bG_{m,k}]\ar[r] & \cM_{cub,k}
%}
%$$
%Using the above conjecture, we have an extra map
%$$\mHH_{\mathrm{tmf}}(X)\to\mHH_{fil}(X)$$
%\begin{remark}
%\textcolor{red}{(???)Even though in this subsection we restrict to the case of algebraically closed fields, we believe a version of this also holds over $\bZ_{p}$ thus fitting in the picture of \cite{MRT}. (???)}
%\end{remark}

%We can construct similarly a TMF variant in the equivariant situation.
%\begin{definition}
%Let $X$ be a smooth quasi-projective variety over $\Spec \bZ$ acted on by a reductive group $G$. The TMF-Hochschild homology of $[X/G]$ is the relative global sections of $$\Mapo{\cE^{un}}{[X/G]}_{/\cM_{ell}}$$
%That is, it is the pushforward of the structure sheaf to the coarse moduli space of relative semistable principal $G$-bundles of degree zero on $\cE^{un}$
%$$\cHH_{\mathrm{TMF}}([X/G]):=p_\ast\cO_{\Mapo{\cE^{un}}{[X/G]}_{/\cM_{ell}}}$$
%where $p$ is the structure map
%$$p:\Mapo{\cE^{un}}{[X/G]}_{/\cM_{ell}}\to\cE^{un}_G$$
%\end{definition}
%\textcolor{red}{I think we should consider only tori}

\providecommand{\bysame}{\leavevmode\hbox to3em{\hrulefill}\thinspace}
\providecommand{\MR}{\relax\ifhmode\unskip\space\fi MR }
\providecommand{\MRhref}[2]{%
  \href{http://www.ams.org/mathscinet-getitem?mr=#1}{#2}
}
\providecommand{\href}[2]{#2}

\end{document}